\documentclass[12pt,reqno]{amsart}
\usepackage{amsmath, amsthm, amssymb}
\usepackage{mathtools}
\usepackage{amsaddr}

\usepackage[margin=1in]{geometry}
\usepackage{parskip}
\usepackage[shortlabels]{enumitem}

\usepackage{xcolor}

\usepackage[bookmarks,colorlinks,breaklinks]{hyperref}
\hypersetup{linkcolor=red,citecolor=blue,
  filecolor=dullmagenta,urlcolor=blue}

\usepackage[numbers]{natbib}
\usepackage{cleveref}

\usepackage{graphicx}
\usepackage{tikz}
\usetikzlibrary{matrix, arrows}
\usetikzlibrary{cd}

\makeatletter
\def\thm@space@setup{%
  \thm@preskip=0.5em\thm@postskip=\thm@preskip%
}
\makeatother

\newtheoremstyle{named}{}{}{\\itshape}{}{\bfseries}{.}{.5em}{\thmnote{#3's }#1}
\theoremstyle{named}

\theoremstyle{plain}
\newtheorem{thm}{Theorem}[section]
\newtheorem{prop}[thm]{Proposition}
\newtheorem{lem}[thm]{Lemma}
\newtheorem{cor}[thm]{Corollary}
\theoremstyle{definition}

\theoremstyle{remark}
\newtheorem{rmk}[thm]{Remark}
\crefformat{footnote}{#2\footnotemark[#1]#3}

\newcommand{\GL}[1]{\mathrm{GL}_{#1}}
\newcommand{\gl}[2]{\mathrm{GL}_{#1}({#2})}
\newcommand{\SL}[1]{\mathrm{SL}_{#1}}
\renewcommand{\sl}[2]{\mathrm{SL}_{#1}({#2})}
\newcommand{\SO}[1]{\mathrm{SO}_{#1}}
\newcommand{\so}[2]{\mathrm{SO}_{#1}({#2})}

\newcommand{\Sp}[1]{\mathrm{Sp}_{#1}}
\renewcommand{\sp}[2]{\mathrm{Sp}_{#1}({#2})}
\newcommand{\GSp}[1]{\mathrm{GSp}_{#1}}
\newcommand{\gsp}[2]{\mathrm{GSp}_{#1}({#2})}
\newcommand{\Spin}[1]{\mathrm{Spin}_{#1}}
\newcommand{\spin}[2]{\mathrm{Spin}_{#1}({#2})}
\newcommand{\GSpin}[1]{\mathrm{GSpin}_{#1}}
\newcommand{\gspin}[2]{\mathrm{GSpin}_{#1}({#2})}

\newcommand{\bkp}{\bar\kappa}

\newcommand{\coh}[3]{H^{#1}({#2},{#3})}

\newcommand{\gal}[2]{\mathrm{Gal}{(#1 / #2)}}
\newcommand{\Gal}[1]{\Gamma_{#1}}
\newcommand{\galQ}{\Gamma_{\rats}}

\newcommand{\rats}{\mathbb{Q}}
\newcommand{\reals}{\mathbb{R}}

\newcommand{\cmplx}{\mathbb{C}}
\newcommand{\ints}{\mathbb{Z}}
\newcommand{\Qp}{\rats_p}
\newcommand{\Zp}{\ints_p}
\newcommand{\Ql}{\rats_l}
\newcommand{\Zl}{\ints_l}
\newcommand{\bQ}{\overline{\mathbb Q}}
\newcommand{\bQp}{\overline\rats_p}
\newcommand{\bZp}{\overline\ints_p}
\newcommand{\bQl}{\overline\rats_l}

\newcommand{\Fp}{\mathbb{F}_p}
\newcommand{\Fl}{\mathbb{F}_l}
\newcommand{\bFp}{\overline{\mathbb{F}}_p}

\newcommand{\bp}{\begin{pmatrix}}
\newcommand{\ep}{\end{pmatrix}}
\newcommand{\mc}{\mathcal}
\newcommand{\mf}{\mathfrak}
\newcommand{\mb}{\mathbb}
\newcommand{\mr}{\mathrm}

\newcommand{\frob}[1]{\mathrm{Fr}_{#1}}
\newcommand{\legndr}[2]{\left(\dfrac{#1}{#2}\right)}

\newcommand{\Zmod}[1]{\mathbb{Z}/#1 \mathbb{Z}}

\newcommand{\rsurj}{\twoheadrightarrow}

\newcommand{\op}[1]{\operatorname{#1}}
\newcommand{\ov}{\overline}
\newcommand{\un}{\underline}

\makeatother
\begin{document}


\title{Motivic Galois representations valued in Spin groups\\
}

\author{Shiang Tang}

\begin{abstract}
Let $m$ be an integer such that $m \geq 7$ and $m \equiv 0,1,7 \mod 8$. We construct strictly compatible systems of representations of $\Gamma_{\mathbb Q} \to \mathrm{Spin}_m(\overline{\mathbb Q}_l) \xrightarrow{\mathrm{spin}} \mathrm{GL}_N(\overline{\mathbb Q}_l)$ that is potentially automorphic and motivic. As an application, we prove instances of the inverse Galois problem for the $\mathbb F_p$--points of the spin groups. For odd $m$, we compare our examples with the work of A. Kret and S. W. Shin (\cite{kret-shin}), which studies automorphic Galois representations valued in $\GSpin{m}$.
\end{abstract}

\maketitle 

\begin{center}
\textit{Dedicated to the memory of my mother, Que Ling Ping, 1954/07/08 - 2019/10/28}
\end{center}

\section{Introduction}

In \cite[8.4]{ser:mot}, Serre asks whether there are motives (over $\rats$,
say), whose motivic Galois groups are equal to a given semisimple (or more generally, reductive) algebraic group $G$. This question and its variants have been studied by many people, including N. Katz, M. Dettweiler, S. Reiter, Z. Yun, S. Patrikis and others; see for example \cite{yun:mot}, \cite{pat:exm} and \cite{bce+}. Most of the results in the literature concern exceptional algebraic groups. In this paper, we study a weaker version of Serre's question for spin groups. We find spin groups interesting because their faithful representations have large dimensions and they do not occur in the \'etale cohomology of smooth projective varieties in any obvious way. In \cite{kret-shin}, Kret and Shin prove the existence of Galois representations into $\GSpin{2n+1}$ corresponding to suitable cuspidal automorphic representations of $\GSp{2n}$ over totally real fields. Their Galois representations are motivic in the sense that they occur in the cohomology of certain Shimura varieties. Studying $\Spin{2n+1}$--valued Galois representations then appears to be a simple matter of passing from a reductive group to its (semisimple) derived subgroup, but the distinction can indeed be subtle: for example, when $n=1$, $\GSpin{3}=\GL{2}$, it is well-known that holomorphic modular forms (or elliptic curves) give rise to $p$-adic Galois representations $\rho: \galQ \to \gl{2}{\bQp}$ that are odd, i.e. $\det \rho(c)=-1$; in particular, they do not land in $\SL{2}$, neither can they be twisted into $\SL{2}$ in any obvious way. In fact, two-dimensional geometric Galois representations that are even (i.e. $\det \rho (c)=1$) are expected to come from Maass forms by the Fontaine--Mazur--Langlands conjectures.

Our main theorem is the following:

\begin{thm}\label{main theorem} (Proposition \ref{char zero spin rep} and Theorem \ref{spin compatible system})
Let $m$ be an integer such that $m \geq 7$ and $m \equiv 0,1,7 \mod 8$. There exists
a strictly compatible system \[R_{\lambda}: \Gal{\rats} \to \gl{N}{\ov{M}_{\lambda}}\] with distinct Hodge--Tate weights and with coefficients in a number field $M$ such that for a density one set of rational primes $l$ and for $\lambda \vert l$, $R_{\lambda}=\mr{spin} \circ r_{\lambda}$, where  $r_{\lambda}: \galQ \to \spin{m}{\ov{M}_{\lambda}}$ is a homomorphism with Zariski-dense image and $\mr{spin}: \Spin{m} \to \GL{N}$ is the spin representation.
Moreover, $\{R_{\lambda}\}$ is potentially automorphic and motivic in the following sense: 
\begin{itemize}
    \item There exist a totally real extension $F^+/\rats$ and a regular L--algebraic, cuspidal automorphic representation $\Pi$ of $\gl{N}{\mb A_{F^+}}$ such that 
$R_{\lambda}|_{\Gal{F^+}} \cong r_{\Pi,\iota_l}$, where $\iota_l: \bQl \xrightarrow{\sim} \cmplx$ is a fixed field isomorphism.
    \item There is a smooth projective variety $X/\rats$ and integers $i$ and $j$ such that $R_{\lambda}$ is a $\galQ$--subrepresentation of $\coh{i}{X_{\bQ}}{\bQl(j)}$.
\end{itemize}
\end{thm}

This is the spin-analog of the main result of \cite{bce+}. As an application, we prove new instances of the inverse Galois problem:
\begin{thm}(Corollary \ref{igp for spin})
For $m \geq 7$, $m \equiv 0,1,7 \mod 8$, $\spin{m}{\Fp}$ is the Galois group of a finite Galois extension of $\rats$ for $p$ belonging to a set of rational primes of positive density.
\end{thm}

Let us explain the strange-looking congruence condition on $m$ appeared in Theorem \ref{main theorem}. Let $T$ be a maximal split torus of $\Spin{m}$ and let $W=N(T)/T$ be the Weyl group. Then the integer $m$ satisfies the above congruence condition if and only if the longest element $w^0 \in W$ acts as $-1$ on $X^{\bullet}(T)$ and $w^0$ has a representative in $N(T)$ of order 2 (\cite[\S 3]{ah:weyl}; see also Lemma \ref{NT sequence for Spin}). Having such an involution at hand is crucial in standard Galois deformation theoretic arguments: one typically constructs an appropriate mod $p$ representation $\bar r: \galQ \to \spin{m}{\bFp}$ which is odd in the sense of \cite[Definition 1.2]{fkp:reldef} (which will hold if the complex conjugation maps into the conjugacy class of the above involution in $\Spin{m}$), then deform it to a geometric (in the sense of Fontaine--Mazur) characteristic zero representation using either Ramakrishna style techniques (originated in \cite{ram02} and sublimed in \cite{fkp:reldef}) or Khare--Wintenberger style arguments (\cite{kw:serre}; see \S 3). Oddness of $\bar r$ is crucial in both methods. 

Due to the automorphic nature of our construction, the Hodge--Tate weights of the compatible system in Theorem \ref{main theorem} are distinct. In contrast, suppose $w^0$ acts as $-1$ on $X^{\bullet}(T)$ but it does not lift to an involution in $N(T)$ (which happens if and only if $m \equiv 3,4,5 \mod 8$), then we do not expect Galois representations $r: \galQ \to \spin{m}{\bQl}$ such that $\mr{spin} \circ r: \galQ \to \gl{N}{\bQl}$ comes from a pure motive with \emph{regular} Hodge structure. In fact, the spin representation $\mr{spin}: \Spin{m} \to \GL{N}$ is valued in $\Sp{N}$ for $m \equiv 3,4,5 \mod 8$ by Lemma \ref{spin rep}. If there were such a motive, let $V$ be its real Hodge structure such that $V \otimes_{\reals} \cmplx=\bigoplus_{p+q=w} V_{p,q}$ with $\ov{V}_{p,q}=V_{q,p}$ and $w \in \ints$. By symmetry of the symplectic torus and regularity, we have $\dim V_{p,q}=\dim V_{-p,-q}=1$ for all $p,q$ appearing in the direct sum. In particular, $w=0$ and the complex conjugation action corresponds to the longest element in the Weyl group of $\Sp{N}$, which lifts to an order 4 element in the normalizer of the torus of $\Sp{N}$, a contradiction.  \footnote{This follows from Lemma 3.1 (with $\delta=1$) and the proof of Lemma 4.12 of \cite{ah:weyl}. One can also verify this directly.} 
Assuming the Fontaine--Mazur conjecture, we can express this in purely Galois theoretic terms: for $m \equiv 3,4,5 \mod 8$, there should not exist $\ell$-adic Galois representations 
\[\galQ \to \spin{m}{\bQl} \xrightarrow{\mr{spin}} \gl{N}{\bQl}\]
that are unramified everywhere, potentially semistable at $l$ with distinct Hodge--Tate weights. For example, let $m=3$, then $\Spin{3} \cong \SL{2}$ and the spin representation is the canonical injection $\SL{2} \xhookrightarrow{} \GL{2}$. In this case, our speculation follows from a theorem of Calegari (\cite[Theorem 1.2]{cal:evenII}) under some mild hypotheses. 


\textbf{Methods and organization of this paper}. The method we use in proving Theorem \ref{main theorem} is very similar to that of \cite{bce+}.
In \S \ref{sec:mod p rep}, we begin by constructing a mod $p$ representation $\bar r: \galQ \to \spin{m}{\bFp}$ (for $m$ satisfying the congruence condition in Theorem \ref{main theorem}) related to the action of
the Weyl group of $\Spin{m}$ on the weight space of the spin representation for which $\mr{spin} \circ \bar r$ satisfies the assumptions of the automorphic lifting theorems in \cite{blggt}. Succeeding in constructing such a representation requires, in addition to the techniques used in \cite[\S 2]{bce+}, a detailed calculation on the structure of $\Spin{m}$ and the crucial observation that a certain part of the Weyl group acts transitively on the weight lattice of the spin representation (Lemma \ref{transitivity}). We then deform $\bar r$ to a geometric representation $r: \galQ \to \spin{m}{\bQp}$ that is Steinberg at a finite place using a version of the Khare--Wintenberger argument: similar arguments have appeared in \cite[\S 3]{bce+} and \cite[\S 3]{pt:gspin}, we present an axiomatized version in \S \ref{sec:lifting gal}. Then \cite[Theorem C]{blggt} implies that $\mr{spin} \circ r$ is potentially automorphic and is part of a strictly compatible system $\{R_{\lambda}\}$. We wish to show that the Zariski closure of the image of $R_{\lambda}$ equals $\Spin{m}$ for \emph{all} $\lambda$. Following \cite[\S 4]{bce+}, we exploit the fact that the image of $r$ contains a regular unipotent element of $\Spin{m}$ and use ideas of Larsen and Pink (\cite{lp:lind}), where the key is to show that $R_{\lambda}$ is irreducible for all $\lambda$. This is proven for $E_6$ in \cite{bce+} using elementary combinatorial properties of the formal character of $E_6$, which does not carry over to $\Spin{m}$ since the rank of the latter can get arbitrarily large. Instead, we get away with a weaker statement by invoking \cite[Theorem D]{blggt} (which relies on Larsen's work \cite{lar:max}): this is where the density one condition in Theorem \ref{main theorem} came from. \footnote{One could conceivably impose additional local deformation conditions (for instance, those corresponding to supercuspidal representations) to force $R_{\lambda}$ to be irreducible for all $\lambda$, but I do not know how to do this at present.} This is done in \S \ref{sec:compatible systems}. In \S \ref{sec:coh}, we show that $\{R_{\lambda}\}$ occurs in the cohomology of a smooth projective variety following \cite[\S 5]{bce+}. In \S \ref{sec:KS}, we compare our construction with the work of Kret and Shin (\cite{kret-shin}): we explain how the main theorem of \cite{kret-shin} yields a stronger version of Theorem \ref{main theorem} for $m \equiv 1,7 \mod 8$. We also observe that for $m \equiv 3,5 \mod 8$, \cite{kret-shin} (in which the automorphic representations are regular L--algebraic) does not yield $\Spin{m}$--valued Galois representations (Lemma \ref{C-algebraic}); moreover, \cite{kret-shin} has no implication on the case when $m \equiv 0 \mod 8$.

\textbf{Acknowledgements}.
I thank Stefan Patrikis for helpful conversations and for suggesting me to compare my construction with \cite{kret-shin}. I thank Patrick Allen for carefully reading through a draft of this paper and making useful comments. 
Last but not least, I thank my mother, who was taken to a better place on October 28, 2019, for her constant encouragements and unconditional support on my mathematical endeavors when I was an undergraduate student in China. Without those, I would not have had the opportunity to earn a doctorate in Mathematics in the United States, much less been equipped to write the present paper. This work is dedicated to the memory of her with enormous gratitude and deepest respect. 

\subsection{Notation}\label{sec:notation}
Let $F$ be a field. Fix an algebraic closure $\ov{F}$ of $F$ and write $\Gamma_F$ for the absolute Galois group $\gal{\ov{F}}{F}$ of $F$. If $F$ is a number field, then for each place $v$ of $F$, we fix an embedding $\ov{F} \to \ov{F_v}$ into an algebraic closure of $F_v$, which gives rise to an injective group homomorphism $\Gamma_{F_v} \to \Gamma_F$. 
For any finite place $v$, let $k_v$ be the residue field of $v$ and let $\frob{v} \in \Gamma_{k_v}$ be the arithmetic Frobenius.
If $H$ is a group (typically the points over a finite field or a $p$-adic field of a reductive algebraic group), and there is a continuous group homomorphism $r: \Gamma_F \to H$, we will sometimes write $r|_v$ 
for $r|_{\Gamma_{F_v}}$, the restriction of $r$ to the decomposition group $\Gamma_{F_v}$. If $H$ acts on a finite-dimensional vector space $V$, we write $r(V)$ for the $\Gamma_F$--module induced by precomposing this action with $r$. (Typically $H$ will be a reductive algebraic group and $V$ will be its Lie algebra equipped with the adjoint action of $H$.) Let $\kappa: \Gamma_F \to \Zp^{\times}$ be the $p$-adic cyclotomic character and $\bkp$ be its reduction modulo $p$. We will always assume $p \neq 2$, and our main theorems will make stronger hypotheses on $p$.  

We recall here some deformation-theoretic terminology. 
Given a topologically finite-generated profinite group $\Gamma$, a finite extension $E/\Qp$ with ring of integers $\mc O$ and residue field $k$, a reductive algebraic group $G$ defined over $\mc O$ and a continuous homomorphism $\bar{r}: \Gamma \to G(k)$, let $R^{\Box}_{\mc{O}, \bar{r}}$ be the universal lifting ring representing the functor sending a complete local noetherian $\mc O$--algebra $R$ with residue field $k$ to the set of lifts $r: \Gamma \to G(R)$ of $\bar{r}$. We will always leave the $\mc{O}$ implicit, writing only $R^{\Box}_{\bar{r}}$, and at various points in the argument we enlarge $\mc{O}$; see \cite[Lemma 1.2.1]{blggt} for a justification of (the harmlessness of) this practice. We write $R_{\bar r}^{\Box} \otimes \bQp$ for $R_{\mc O, \bar r}^{\Box} \otimes_{\mc O} \bQp$ for any particular choice of $\mc{O}$, and again by \cite[Lemma 1.2.1]{blggt}, $R_{\bar r}^{\Box} \otimes \bQp$ is independent of the choice of $E$.

When $K/\Qp$ is a finite extension and $\Gamma$ is $\Gamma_K$, we consider quotients of $R^{\Box}_{\bar{r}}$ having fixed inertial type and $p$-adic Hodge type. The fundamental analysis here is due to Kisin (\cite{kis:pst}), and the state of the art, and our point of reference, is \cite{bg:Gdef}, and we refer there for details. We will index $p$-adic Hodge types of deformations $r \colon \Gamma_K \to G(\mc{O})$ by collections $\underline{\mu}(r)=\{\mu(r, \tau)\}_{\tau \colon K \to \bQp}$ of (conjugacy classes of) Hodge--Tate co-characters, and write $R^{\Box, \underline{\mu}(r)}_{\bar{r}}$ for the $\Zp$-flat quotient of $R^{\Box}_{\bar{r}}$ whose points in finite local $E$--algebras are precisely those of $R^{\Box}_{\bar{r}}$ that are moreover potentially semi-stable with $p$-adic Hodge type $\underline{\mu}(r)$. We likewise consider the quotients with fixed inertial type $\sigma$, $R^{\Box, \un{\mu}(r), \sigma}_{\bar{r}}$, referring to \cite[\S 3.2]{bg:Gdef} for details.

We recall some basic facts on spin groups and spin representations, see \cite[Lecture 20]{fh:rep} for more details. Let $m \geq 3$ be an integer. Consider the symmetric form 
$$(x,y)=x_1y_{m}+x_2y_{m-1}+\cdots+x_{m}y_1$$ on $V=\rats^{m}$ with associated quadratic form $Q(x)=(x,x)$. Let $C(Q)$ be the Clifford algebra associated to $(V,Q)$. It is equipped with an embedding $V \subset C(Q)$ which is universal for maps $f: V \to A $ into associative rings $A$ satisfying $f(x)^2=Q(x)$ for all $x \in V$. The algebra $C(Q)$ has a $\Zmod{2}$--grading, $C(Q)=C(Q)^+\oplus C(Q)^-$, induced from the grading on the tensor algebra. On the Clifford algebra $C(Q)$ we have an unique anti-involution * that is determined by $(v_1\cdots v_r)^*=(-1)^rv_r\cdots v_1$ for all $v_1, ... ,v_r \in V$. We define for all $\rats$--algebras $R$, $$\gspin{m}{R}=\{ g \in (C(Q)^+ \otimes R)^{\times}: gVg^* \subset V \}.$$ Let $N: C(Q) \to \rats^{\times}$, $x \to xx^{*}$ be the Clifford norm. It induces a character $N: \GSpin{m} \to \mb G_m$. We define $\Spin{m}$ to be the kernel of $N$. The action of the group $\GSpin{m}$, resp. $\Spin{m}$ stabilizes $V \subset C(Q)$, which induces a surjection $\GSpin{m} \rsurj \SO{m}$, resp. $\Spin{m} \rsurj \SO{m}$. Denote by $\mu_2=\{ \pm 1\}$ the kernel of $\Spin{m} \rsurj \SO{m}$.

Now we turn to spin representations. We keep the notation in the paragraph above with $\rats$ replaced by $\cmplx$. Let $\mf{so}_m(\cmplx):=\mf{so}(Q)=\{X \in \mr{End}(V): (Xv,w)+(v,Xw)=0, \forall v,w \in V \}$. By \cite[Lemma 20.7]{fh:rep}, there is a natural embedding of Lie algebras $\mf{so}(Q) \to C(Q)^{even}$. Let $n=[m/2]$. The weight lattice of $\mf{so}_{m}\cmplx$ may be described as the submodule of $\sum_{1 \leq i \leq n} \rats \chi_i$ spanned by $\chi_1, ... ,\chi_n$ and $\frac{1}{2}(\sum_i \chi_i)$. 
Suppose $m=2n$ is even, write $V=W \oplus W'$, where $W$ and $W'$ are isotropic subspaces of $(V,Q)$. By \cite[Lemma 20.9]{fh:rep}, there is an isomorphism of algebras $C(Q) \cong \mr{End}(\bigwedge^*W)$. We have $\bigwedge^*W=\bigwedge^{even}W \oplus \bigwedge^{odd}W$, it follows that $\mf{so}(Q) \subset C(Q)^{even} \cong \mr{End}(\bigwedge^{even}W)\oplus \mr{End}(\bigwedge^{odd}W)$. Hence we have two representations of $\mf{so}_{m}\cmplx$, $\bigwedge^{even}W$ and $\bigwedge^{odd}W$. We let $S=\bigwedge^{even}W$ for even $n$, and let $S=\bigwedge^{odd}W$ for odd $n$. By \cite[Lemma 20.15]{fh:rep}, $S$ is irreducible of dimension $2^{n-1}$ whose weights are $\frac{1}{2}(\sum_i \pm \chi_i)$ where the number of minus signs is even. We call $S$ the spin representation of $\mf{so}_m(\cmplx)$. Now suppose $m=2n+1$ is odd, write $V=W \oplus W' \oplus U$, where $W$ and $W'$ are isotropic subspaces of $(V,Q)$, and $U$ is a one-dimensional subspace perpendicular to them. In this case, we have $\mf{so}(Q) \subset C(Q)^{even} \cong \mr{End}(\bigwedge^*W)$. We let $S=\bigwedge^*W$. By \cite[Proposition 20.20]{fh:rep}, $S$ is an irreducible representation of $\mf{so}_{m}\cmplx$ of dimension $2^n$ whose weights are $\frac{1}{2}(\sum_i \pm \chi_i)$. We again call $S$ the spin representation of $\mf{so}_m(\cmplx)$. The spin representation may be viewed as a group homomorphism $\mr{spin}: \Spin{m} \to \GL{2^n}$, resp. $\mr{spin}: \Spin{m} \to \GL{2^{n-1}}$ for odd $m$, resp. even $m$. The following fact is contained in \cite[Exercise 20.38]{fh:rep}.

\begin{lem}\label{spin rep}
There is nondegenerate bilinear form $\beta$ on $S$ such that if $m=2n+1$, $\beta$ is symmetric when $n \equiv 0,3 \mod 4$ and skew-symmetric otherwise; and if $m=2n$ with $n$ even, $\beta$ is symmetric when $n \equiv 0 \mod 4$ and skew-symmetric otherwise. 
Therefore, if $m=2n+1$, the spin representation is valued in $\SO{2^n}\cmplx$ for $n \equiv 0,3 \mod 4$ and valued in $\Sp{2^n}\cmplx$ for $n \equiv 1,2 \mod 4$; and if $m=2n$ with $n$ even, the spin representation is valued in $\SO{2^{n-1}}\cmplx$ for $n \equiv 0 \mod 4$ and valued in $\Sp{2^{n-1}}\cmplx$ for $n \equiv 2 \mod 4$.
\end{lem}

\section{Construction of residual Galois representations}\label{sec:mod p rep}

Before we construct our mod $p$ Galois representations into $\Spin{m}$, we make some preliminary calculations. We follow the notation of \S 1.1. 

We have the following explicit description of the two-fold covering 
$\pi: \Spin{m} \rsurj \SO{m}$ (see the proof of \cite[Proposition 20.28]{fh:rep}). Let $g \in \SO{m}=\mr{SO}(Q)$, we may write $g$ as a product of reflections $r_{w_1}\cdots r_{w_r}$ for $w_i \in V$ with $Q(w_i)=-1$ ($r$ is necessarily even), then the two elements in $\pi^{-1}(g)$ are $\pm w_1\cdots w_r$. We fix the following ordered basis for $V=\rats^m$: for $m=2n$, $e_1, ... ,e_n,f_n, ... ,f_1$ satisfying $(e_i,f_i)=1$ for $1\leq i \leq n$, $(e_i,f_j)=0$ for $1 \leq i \neq j \leq n$, and $(e_i,e_j)=(f_i,f_j)=0$ for $1 \leq i,j \leq n$; for $m=2n+1$, $e_1, ... ,e_n, u_0, f_n, ... ,f_1$ satisfying $(e_i,f_i)=1$ for $1\leq i \leq n$, $(e_i,f_j)=0$ for $1 \leq i \neq j \leq n$, $(e_i,e_j)=(f_i,f_j)=0$ for $1 \leq i,j \leq n$, $(u_0,u_0)=1$, $(u_0,e_i)=(u_0,f_i)=0$ for $1\leq i \leq n$. With this basis, we have for any $\rats$--algebra $R$
$$\SO{m}(R)=\{ g \in \sl{m}{R}: gJg^t=J \}$$
where $J=[a_{ij}]_{1 \leq i,j \leq m}$ whose nonzero entries are exactly $a_{i,m+1-i}=1$ for $1 \leq i \leq m$. Then we have a standard maximal torus $T_{SO}$ defined by $T_{SO}(R)=\{diag(t_1, ... ,t_n,1,t_n^{-1}, ... ,t_1^{-1}): t_i \in R^{\times} \}$ for any $\rats$--algebra $R$, where we omit the middle entry 1 when $m=2n$ is even. We let $T=\pi^{-1}T_{SO}$ be the corresponding maximal torus of $\Spin{m}$. The Weyl group $W$ is then isomorphic to $N(T_{SO})/T_{SO}$, which is a semidirect product $D \rtimes S_n$, where $D$ is described in what follows. 
We let $D$ be $\{\pm 1\}^n$ when $m=2n+1$, and the subgroup of $\{\pm 1\}^{n}$ consisting of elements with an even number of minus signs when $m=2n$. It acts on $T_{SO}$ in the following way: For any element $\epsilon=(\epsilon_i) \in D$ and $t=diag(t_1, ... ,t_n,1,t_n^{-1}, ... ,t_1^{-1}) \in T_{SO}$, define $t^{\epsilon}:=diag(t_1^{\epsilon_1}, ... ,t_n^{\epsilon_n},1,t_n^{-\epsilon_n}, ... ,t_1^{-\epsilon_1}) \in T_{SO}$. The symmetric group $S_n$ acts on $T_{SO}$ by permuting the diagonal entries: $\forall \sigma \in S_n$, $t^{\sigma}:=diag(\sigma(t_1), ... ,\sigma(t_n),1,\sigma(t_n)^{-1}, ... ,\sigma(t_1)^{-1})$. 

Let $w^0$ be the longest element in $W$. When $m=2n+1$, $w^0$ acts as $-1$ on $X^{\bullet}(T_{SO})$; when $m=2n$, $w^0$ acts as $-1$ on $X^{\bullet}(T_{SO})$ if and only if $n$ is even. In both cases, $w^0=(-1,...,-1) \in D$.

\begin{lem}\label{transitivity}
The group $D \leq W$ acts simply transitively on $\Lambda_{spin}:=\{ \frac{1}{2}(\sum_{i=1}^n \pm\chi_i) \}$, the weight lattice of the spin representation of $\Spin{m}$.
\end{lem}
\begin{proof}
This follows immediately from the explicit action of $D$ on $T$ described above.
\end{proof}

\begin{lem}\label{NT sequence for SO}
The group homomorphism $N(T_{SO}) \rsurj W$ admits a section $s: W \to N(T_{SO})$. Indeed, for $\sigma \in S_n$, let $M_{\sigma}$ be the corresponding standard $n \times n$ permutation matrix, then $s(\sigma)=\bp M_{\sigma}&0\\0&M_{\sigma}^{at} \ep$, where $M_{\sigma}^{at}$ is the anti-diagonal transpose of $M_{\sigma}$; for $\epsilon=(\epsilon_i) \in D$, $s(\epsilon)=d_{\epsilon}$, where $d_{\epsilon}$ is the $m \times m$ matrix obtained from the identity matrix $I_m$ by swapping the $i$-th and the $(m+1-i)$-th columns whenever $\epsilon_i=-1$ and leaving the rest of the columns unchanged when $m=2n$ is even, and $d_{\epsilon}$ is the $m \times m$ matrix obtained from the identity matrix $I_m$ by performing the column operations above and replacing the $(n+1,n+1)$--entry with $(-1)^{|\{1 \leq i \leq n: \epsilon_i=-1\}|}$ when $m=2n+1$ is odd. 
\end{lem}

By the lemma above, we may identify $D$ with its isomorphic image in $N(T_{SO})$. With this identification, let $\tilde D=\pi^{-1}D \subset \Spin{m}$. We would like to calculate the structure of $\tilde D$. Suppose $m=2n$ is even, for each $i \in [1,n]$, let $w_i=\frac{1}{\sqrt{2}}(e_i-f_i)$. A straightforward calculation shows that $Q(w_i)=-1$ and $r_{w_i} \in \mr{O}_m$ equals the matrix obtained by swapping the $i$-th and the $(m+1-i)$-th columns of $I_m$. It follows that $D$ is generated by $r_{w_i}r_{w_j}$ for $1 \leq i < j \leq n$, and hence $\tilde D$ is generated by $w_iw_j$ for $1 \leq i < j \leq n$. 
Note that $w_i^2=-1$ and $w_iw_j=-w_jw_i$ for $i \neq j$, which imply that 
$$(w_1\cdots w_k)^2=(-1)^{k(k+1)/2}$$
for $1 \leq k \leq n$. Now suppose $m=2n+1$ is odd, define $w_i$ as above and let $w_0:=\sqrt{-1}u_0$, then $r_{w_0}r_{w_i} \in \SO{m}$ equals the matrix obtained from $I_m$ by swapping the $i$-th and the $(m+1-i)$-th columns and replacing the $(n+1,n+1)$--entry with $-1$. It follows that $r_{w_0}r_{w_i} $, $1\leq i \leq n$ generate $D$ and hence $\omega_i:=w_0w_i$, $1\leq i \leq n$ generate $\tilde D$. Note that $\omega_i^2=-1$ and $\omega_i\omega_j=-\omega_j\omega_i$ for $i \neq j$, which imply that 
$$(\omega_1\cdots\omega_k)^2=(-1)^{k(k+1)/2}$$
for $1 \leq k \leq n$. We thus obtain

\begin{lem}\label{NT sequence for Spin}
The sequence $1 \to \{ \pm 1\} \to \tilde D \to D \to 1 $ is nonsplit. Moreover, if $m=2n$, $w^0$ acts as $-1$ on $X^{\bullet}(T)$ if and only if $n$ is even, in which case it corresponds to the element $w_1\cdots w_n \in \Spin{m}$ under the map $N(T) \rsurj W$. This element has order two if and only if $n \equiv 0 \mod 4$. If $m=2n+1$, $w^0$ always acts as $-1$ on $X^{\bullet}(T)$, and the corresponding element $\omega_1\cdots\omega_n \in \Spin{m}$ has order two if and only if $n \equiv 0,3 \mod 4$.
\end{lem}

For the rest of this section, we make the following \emph{parity assumption} on the integer $m$: 
\begin{itemize}
    \item If $m=2n$ is even, then $n \equiv 0 \mod 4$;
    \item If $m=2n+1$ is odd, then $n \geq 3$ and $n \equiv 0,3 \mod 4$; 
\end{itemize}
equivalently, $m \geq 7$ and $m\equiv 0,1,7 \mod 8$.

By Lemma \ref{NT sequence for Spin}, $w^0 \in W$ then lifts to an order two element in $N(T)$; and it follows from Lemma \ref{spin rep} that the spin representation is $$\mr{spin}: \Spin{m} \to \SO{N}$$ where $N=2^{n}$ for odd $m$ and $N=2^{n-1}$ for even $m$. 

We now begin to construct the mod $p$ representation valued in $\Spin{m}$. We do this by first realizing $\tilde D$ as a Galois group of some finite extension of $\rats$, and then modify the corresponding homomorphism $\galQ \to N(T)_{\Fp} \subset \Spin{m}(\Fp)$ so that after composing with $\mr{spin}$, the resulting representation satisfies the assumptions of the potential automorphy theorems of \cite{blggt}.

\begin{lem}
There is a finite, totally real Galois extension $L/\rats$ whose Galois group is isomorphic to $D$ satisfying property $(S_N)$ with $N=2$ in the sense of \cite[Definition 2.1.2]{ser:igp}, i.e. every prime $p$ which is ramified in $L/\rats$ satisfies
\begin{itemize}
    \item $p \equiv 1 \mod 4$.
    \item If $v$ is a place of $L$ above $p$, then the local extension $L_v/\Qp$ is totally ramified.
\end{itemize}
\end{lem}
\begin{proof}
We have $D \cong (\Zmod{2})^n$ or $(\Zmod{2})^{n-1}$. We pick a prime $p_1 \equiv 1 \mod 4$. We will construct inductively $p_1, ... ,p_n$ such that $p_i \equiv 1 \mod 4$ and $p_i \in (\rats_{p_j}^{\times})^2$ for $i \neq j$. Suppose we have $p_1, ... ,p_k$ ($k<n$) already and consider $\rats(\sqrt{-1},\sqrt{p_1}, ... ,\sqrt{p_k})$. Chebotarev implies that there are infinitely many $p$ that splits in this field. In particular, $-1, p_1, ... ,p_k \in (\Qp^{\times})^2$, which is equivalent to 
    $$\legndr{-1}{p}=\legndr{p_1}{p}=\cdots=\legndr{p_k}{p}=1.$$
    In particular, $p \equiv 1 \mod 4$ and quadratic reciprocity implies that $$\legndr{p}{p_1}=\cdots=\legndr{p}{p_k}=1,$$ which is equivalent to 
    $p \in (\rats_{p_j}^{\times})^2$ for $1 \leq j \leq k$. Thus, if we take $p_{k+1}$ to be such a prime, the induction can proceed. Now we let $L=\rats(\sqrt{p_1}, ... ,\sqrt{p_n})$. Then $\gal{L}{\rats}=D$ and $L_{v_i}=\rats_{p_i}(\sqrt{p_i})$, which is a totally ramified extension of $\rats_{p_i}$.
\end{proof}

Let $\underline s: \galQ \rsurj \gal{L}{\rats}=D$ be the corresponding homomorphism produced by the above lemma. We need to lift $\underline s$ to a surjective homomorphism $s: \galQ \to \tilde D$ across the short exact sequence $1 \to \mu_2 \to \tilde D \to D \to 1$. The method we use is a minor modification of the argument in page 12 of \cite{ser:igp}.  
Since the exact sequence of groups is nonsplit by Lemma \ref{NT sequence for Spin}, it suffices to find a lift $s: \galQ \to \tilde D$ of $\underline s$, which will automatically be surjective. Suppose that $\underline s$ has a lift to $\tilde D$ everywhere locally (including $\infty$), then $\underline s$ has a lift to $\tilde D$. In fact, $\underline s$ induces a map $s^*: \coh{2}{D}{\mu_2} \to \coh{2}{\galQ}{\mu_2}$, and $\underline s$ lifts to $\tilde D$ if and only if $\underline{s}^*(\xi)=0$, where $\xi$ is the class of the extension $\tilde D \rsurj D$. The natural map 
$$\coh{2}{\galQ}{\mu_2} \to \prod_{p \leq \infty} \coh{2}{\Gamma_{\Qp}}{\mu_2}$$ is injective because $\coh{2}{\galQ}{\mu_2}=\mr{Br}(\rats)[2]$ and $\mr{Br}(\rats) \xhookrightarrow{} \prod_{p \leq \infty} \mr{Br}(\Qp)$. It follows that if $\underline s$ lifts everywhere locally, then it lifts globally. It thus remains to show that $\underline s$ lifts everywhere locally. For $p<\infty$, $\underline s|_{\Gamma_{\Qp}}$ lifts to $\tilde D$ follows from property $(S_2)$ by the exact same argument as in the last two paragraphs of page 12 of \cite{ser:igp}; for $p=\infty$, $\underline s|_{\Gamma_{\reals}}$ obviously lifts to $\tilde D$ because $\underline s(c)=1$ ($c$ is the complex conjugation). Therefore, we obtain a surjection
\[s: \galQ \rsurj \tilde D \subset N(T)_{\ints}.\]

We would like to modify $s$ so that $s(c)$ belongs to $T$ and is `sufficiently odd' in the sense that $\dim (\mf{so}_{m})^{\mr{Ad}s(c)}=|\Phi|/2$, where $\Phi$ is the root system of $\mf{so}_m$. When $s(c) \in \Spin{m}$ satifies this equality, we call it a \emph{split Cartan involution}, or a \emph{Chevalley involution} in $\Spin{m}$. 

Let us fix a Borel subgroup $B \supset T$ of $\Spin{m}$. It corresponds to a set of simple roots $\Delta \subset \Phi:=\Phi(\Spin{m},T)$. 
Let $\rho^{\vee}$ be the half sum of positive coroots in $\Phi^{\vee}$. By examining the Plates in \cite{bou:lie}, if $\Phi$ is of type $B_n$, $\rho^{\vee}$ has coefficients in $\ints$ if and only if $n \equiv 0,3 \mod 4$; and if $\Phi$ is of type $D_n$, $\rho^{\vee}$ has coefficients in $\ints$ if and only if $n \equiv 0,1 \mod 4$. In particular, $\rho^{\vee}$ has coefficients in $\ints$ for $m=2n$ or $2n+1$ satisfying the parity assumption we made before. It follows that for those $m$, $\rho^{\vee}\in X_{\bullet}(T)$, i.e. it is a well-defined cocharacter of $T$.

Let $k$ be a field and $G$ be an algebraic group defined over $k$, we denote by $G_k$ or $G(k)$ the $k$--points of $G$. Observe that if $s: \galQ \to N(T)_k$ is a group homomorphism and $\phi \in \coh{1}{\galQ}{T_k}$, where $\galQ$ acts on $T_k$ via the composite $\mr{Ad} \circ s$, then $\phi\cdot s: \galQ \to N(T)_k$ is a group homomorphism.

Recall that we have an isomorphism $\mb G_m^n \xrightarrow{\sim} T_{SO}$, $(t_1, ... ,t_n) \mapsto diag(t_1, ... ,t_n,1,t_n^{-1}, ... ,t_1^{-1})$ (where we omit the middle 1 when $m$ is even). We also have an isomorphism 
$$T \xrightarrow{\sim} \{ (z,t_1, ... ,t_n) \in \mb G_m^{n+1}: z^2=t_1\cdots t_n \}.$$
Let $T[2]$ be the group of $\ints$--points of $T$. It consists of elements of the form
$(z,\epsilon_1, ... ,\epsilon_n)$, where $z=\pm 1$, $\epsilon_i=\pm 1$ and the number of $-1$ among $\epsilon_i$ is even.
The group $\galQ$ acts on $T[2]$ via $\mr{Ad} \circ s$. This action factors through $\gal{L}{\rats} \cong D$. Then $T[2]$ is a direct sum of two irreducible $D$--modules, $V_1=\{(1,\epsilon_1, ... ,\epsilon_n): \epsilon_i=\pm 1\}$ and $V_2=\{(-1,\epsilon_1, ... ,\epsilon_n): \epsilon_i=\pm 1\}$. Letting $S=\mr{cs}(L/\rats) \cup \{\infty, 2\}$ ($\mr{cs}(L/\rats)$ denotes the set of places of $\rats$ that split completely in $L$), \cite[Theorem 9.2.3,(v)]{nsw:coh} implies that for $i=1,2$,
the natural maps
$$\coh{1}{\Gamma_{\rats,S}}{V_i} \to \coh{1}{\Gamma_{\reals}}{V_i}$$ are surjective, and hence we have a surjection 
$$\coh{1}{\Gamma_{\rats,S}}{T[2]} \rsurj \coh{1}{\Gamma_{\reals}}{T[2]}=\mr{Hom}(\Gamma_{\reals},T[2])$$
(recall that $s(c) \in \mu_2$).
Let $\phi \in \coh{1}{\Gamma_{\rats,S}}{T[2]}$ be a class that maps to $c \mapsto s(c)\rho^{\vee}(-1)$ in $\mr{Hom}(\Gamma_{\reals},T[2])$. We now replace $s$ by $\phi \cdot s$. Then \emph{having made this replacement}, $s(c)=\rho^{\vee}(-1)$, which is a split Cartan involution by \cite[Lemma 2.3]{yun:mot}. Let $E$ be the fixed field of $\ker s$.

Now we are going to further modify $s$ by an element in $\coh{1}{\galQ}{T_{\Fp}}$ (for a prime $p$) so that the resulting homomorphism $\galQ \to N(T)_{\Fp} \subset \spin{m}{\Fp}$ satisfies favorable conditions so that we can use the results of \cite{blggt} to deform it to a geometric representation. The method we use here is extremely close to the proof of \cite[Proposition 2.8]{bce+}. We first establish some notation. For a fixed prime $p$, we write $\kappa$ for the $p$-adic cyclotomic character and $\bkp$ for its mod $p$ reduction. For a prime $l$ dividing $p-1$ such that $l^2$ does not divide $p-1$, let $\mr{pr}(l)$ be the canonical projection from $\Fp^{\times}$ onto the $l$--torsion subgroup $\Fp^{\times}[l]$. Let $\bkp[l]=\mr{pr}(l)\circ \bkp: \galQ \to \Fp^{\times}[l]$.

\begin{prop}\label{mod p rep}
Consider pairs of primes $(l,p)$ such that $p$ splits in $E/\rats$ and $p-1$ is divisible by $l$ but not $l^2$. Then there exist infinitely many primes $l$ such that there exist infinitely many
pairs $(l,p)$ such that there exists a homomorphism
$$\bar r: \galQ \to T_{\Fp}[l]\cdot \tilde D \subset N(T)_{\Fp}$$
safisfying 
\begin{enumerate}
    \item $\bar r(c)=\rho^{\vee}(-1)$.
    \item At $p$, fix any choice of integers $n_{\alpha}$ for $\alpha \in \Delta$, $\bar r|_{p}=\prod_{\alpha \in \Delta} \alpha^{\vee} \circ \bkp[l]^{n_{\alpha}}$.
    \item There is a prime $q$ of order $l$ mod $p$ such that $\bar r|_q$ is unramified with Frobenius mapping to $\rho^{\vee}(q) \in T_{\Fp}[l]$.
    
    Moreover, in addition to the above, we can choose $p>2(N+1)$, $l \geq h_{\Spin{m}}$ ($h_{\Spin{m}}$ is the Coxeter number of $\Spin{m}$), $\{n_{\alpha}\}$ and $\bar r$ such that
    \item $\mr{spin} \circ \bar r|_p$ is a direct sum of distinct powers of $\bkp[l]$.
    \item For any Borel $B \supset T$, $\coh{0}{\Gamma_{\Qp}}{\bar r(\mf g/\mf b)}$ vanishes (where $\mf g=\mf{so}_m$ and $\bar r$ acts on $\mf g/\mf b$ via the adjoint action).
    \item $\mr{spin} \circ \bar r|_{\Gamma_{\rats(\mu_p)}}$ is absolutely irreducible.
\end{enumerate}
\end{prop}
\begin{proof}
Let $l$ be any odd prime such that $E$ (the fixed field of $\ker s$) and $\rats(\mu_{l^2})$ are linearly disjoint over $\rats$ (e.g. take $l$ split in $E/\rats$). Later in the argument we will require $l$ to be larger than some constant depending only on the group $\Spin{m}$. Now we take $p$ to be any prime split in $E(\mu_l)$ but nonsplit in $E(\mu_{l^2})$ (such prime exists by Chebotarev). In other words, $p$ splits in $E/\rats$ and $p-1$ is divisible by $l$ but not $l^2$. In particular, $E$ and $\rats(\mu_p)$ are linearly disjoint over $\rats$. By Chebotarev, there exists a prime $q$ that splits in $E$ and has order $l$ in $\Fp^{\times}$. Let $r$ be a prime that splits in $E(\mu_p)$. The Galois group $\galQ$ acts on $T_{\Fp}[l]$ via $\mr{Ad} \circ s$, which factors through $D \cong \gal{L}{\rats}$. Then $T_{\Fp}[l]$ decomposes into a direct sum of irreducible $\Fl[D]$--modules: $T_{\Fp}[l]=\bigoplus W_i$. Let $\Sigma=\mr{cs}(L/\rats) \cup \{\infty, l\}$ and $T=\{p,q,r\}$. Then \cite[Theorem 9.2.3,(v)]{nsw:coh} implies that the canonical homomorphisms
$$\coh{1}{\Gamma_{\rats,\Sigma}}{W_i} \to \bigoplus_{v\in T}\coh{1}{\Gamma_{\rats_v}}{W_i}$$ are surjective, from which it follows that
the canonical homomorphism 
$$\coh{1}{\Gamma_{\rats,\Sigma}}{T_{\Fp}[l]} \to \bigoplus_{v\in T}\coh{1}{\Gamma_{\rats_v}}{T_{\Fp}[l]}$$ is surjective as well.
We let $\phi \in \coh{1}{\Gamma_{\rats,\Sigma}}{T_{\Fp}[l]}$ be a class such that 
$\phi|_p=\prod_{\alpha \in \Delta} \alpha^{\vee} \circ \bkp[l]^{n_{\alpha}}$, $\phi|_q$ is unramified with $\phi(\frob{q})=\rho^{\vee}(q)$, and $\phi|_r$ is unramified with $\phi(\frob{r})=t$, where $t$ is any element in $T_{\Fp}[l]$ such that $\lambda(t)$ for $\lambda \in \Lambda_{spin}$ are all distinct (such $t$ exists provided that $l$ is sufficiently large). We set $\bar r=\phi \cdot s$. It follows immediately from the choice of $\phi$ that (2) and (3) hold. Since $s(c)=\rho^{\vee}(-1)$ and $l$ is odd, (1) holds. It remains to address (4) through (6). We have
$$\mr{spin} \circ \bar r|_p=\bigoplus_{\lambda \in \Lambda_{spin}} \bkp[l]^{\sum_{\alpha \in \Delta} n_{\alpha}<\lambda,\alpha^{\vee}> };$$
since $\Lambda_{spin}$ is multiplicity-free, choosing $\{n_{\alpha}\}$ to be sufficiently general will make the exponents distinct. This proves (4). For (5), note that
$$\bar r|_p(\mf g/\mf b)=\bigoplus_{\beta \in \Phi^-} \bkp[l]^{\sum_{\alpha \in \Delta} n_{\alpha}<\beta,\alpha^{\vee}> }.$$
By choosing $l$ to be sufficiently large and $\{n_{\alpha}\}$ to be sufficiently general, we can ensure that the absolute values of the exponents are in $(0,l)$; in particular, $\coh{0}{\Gamma_{\Qp}}{\bar r(\mf g/\mf b)}$ vanishes. Finally, $D$ acts transitively on $\Lambda_{spin}$ by Lemma \ref{transitivity}, so any nonzero submodule of $\mr{spin} \circ \bar r$ has nonzero projection to each of the weight spaces; since $E$ and $\rats(\mu_p)$ are linearly disjoint over $\rats$, the same is true for $\mr{spin} \circ \bar r|_{\Gamma_{\rats(\mu_p)}}$. On the other hand, $\bar r(\frob{r})=t \in T_{\Fp}[l]$ acts via distinct characters on the different weight spaces of the spin representation by the choice of $t$. It follows that $\mr{spin} \circ \bar r|_{\Gamma_{\rats(\mu_p)}}$ is absolutely irreducible, which shows (6).
\end{proof}

\section{Lifting Galois representations}\label{sec:lifting gal}
In this section, we prove a lifting theorem using a version of the Khare--Wintenberger argument (\cite{kw:serre}), similar to the arguments in \cite[\S 3]{bce+} and \cite[\S 3]{pt:gspin}. We axiomatize their arguments and state our lifting theorem as general as we can. On the other hand, the main theorem of \cite{fkp:reldef} allows one to deform a much larger class of $G$--valued mod $p$ Galois representations to geometric representations in characteristic zero, yet it does not establish potential automorphy of the lifts.

Let $G$ be a split connected semisimple group over $\rats$. Let $T$ be a maximal split torus of $G$ and let $\Phi=\Phi(G,T)$ be the root system of $G$. We assume that $G$ satisfies the followings:
\begin{itemize}
    \item $G$ contains a split Cartan involution, i.e. there is an element $\tau \in G$ of order two such that $\dim \mf g^{\mr{Ad} \tau}=|\Phi|/2$.
    \item $G$ admits an irreducible faithful representation $R: G \to \GL{N}$ factoring through $\SO{N}$ (the orthogonal group preserving the pairing $(x,y)=x_1y_1+x_2y_{2}+\cdots+x_Ny_N$ on $\rats^N$; this is different from the pairing used in \S \ref{sec:notation}) such that the formal character of $R$ is multiplicity-free. \footnote{This is a very restrictive assumption.}
\end{itemize}

Let $F^+$ be a totally real field. Let $\bar r: \Gamma_{F^+} \to G(\bFp)$ be a continuous representation. Let $S$ be the union of all the archimedean places of $F^+$, the places of $F^+$ dividing $p$ and the places of $F^+$ at which $\bar r$ is ramified. Suppose that $\bar r$ satisfies the followings:
\begin{itemize}
    \item $\bar r$ is odd in the sense that for $v|\infty$ with $c_v \in \Gamma_{F^+_v}$ the complex conjugation, $\bar r(c_v)$ is a split Cartan involution.
    \item For $v$ not above $p$, $\bar r|_v$ has a characteristic zero lift $r_v: \Gamma_{F^+_v} \to G(\bZp)$.
    \item For $v|p$, $\bar r|_v$ has a characteristic zero lift $r_v: \Gamma_{F^+_v} \to G(\bZp)$ with some fixed Hodge type $\un{\mu}(r_v)$ and inertial type $\sigma_v$ such that $R \circ r_v$ is potentially diagonalizable in the sense of \cite{blggt} with regular Hodge--Tate weights. \footnote{For the regularity to hold, it is necessary that the formal character of $R$ is multiplicity-free.}
    \item $R \circ \bar r|_{F^+(\mu_p)}$ is absolutely irreducible.
    \item $p$ is sufficiently large relative to $G$.
\end{itemize}

We let $\mc O$ be the ring of integers in a finite extension $E$ of $\Qp$, enlarged if necessary so that all of the above data are defined over $\mc O$. 

We recall the definition of the Clozel--Harris--Taylor group scheme $\mc G_n$ over $\ints$ which is defined as the semidirect product $(\GL{n} \times \GL{1}) \rtimes \{1,\jmath\}$ where $\jmath(g,a)\jmath=(a({}^t g)^{-1},a)$, and the similitude character $\nu: \mc G_n \to \GL{1}$ given by $\nu(g,a)=a$ and $\nu(\jmath)=-1$. Suppose we have a homomorphism $r \colon \Gamma_{F^+} \to G(A)$ for some ring $A$. 
Let $F/F^+$ be a quadratic extension of $F$, and define $\rho(r) \colon \Gamma_{F^+} \to \mc{G}_{N}(R)$ as the composite
\[
\Gamma_{F+} \xrightarrow{R(r) \times \mr{res}_{F}} \mr{SO}_{N}(A) \times \mr{Gal}(F/F^+) \to \mc{G}_{N}(A),
\]
where the last map sends $g\in \mr{SO}_{N}(A)$ to its image in $\gl{N}{A}$ and sends the nontrivial element of $\gal{F}{F^+}$ to $\jmath$. 
By our choice of the pairing defining $\SO{N}$, $\rho(r)$ is a well-defined homomorphism. 

We choose a quadratic CM extension $F/F^+$ such that $F$ does not contain $\zeta_p$, $R \circ \bar r|_{F(\mu_p)}$ remains irreducible, and all the finite places in $S$ split in $F/F^+$. 
We will define global deformation conditions for $\bar r$ and $\rho(\bar{r})$ with respect to $F/F^+$. 
Then we will show that there are natural finite maps between the corresponding deformation rings and use the $\mc O$--finiteness of the deformation ring for the group $\mc G_{N}$ 
to conclude that the deformation ring for the group $G$ is $\mc O$--finite; this, combined with a standard calculation of its Krull dimension (as in \cite{bg:Gdef}), will imply that it has a $\bQp$--point. For the $\mc G_N$--deformation problems, we will fix the multiplier: Let $\delta_{F/F^+}: \Gamma_{F^+} \to \{ \pm 1\}$ be the quadratic character associated with $F/F^+$; we require that any local or global deformation of $\rho(\bar r)$ composed with $\nu$ equals $\delta_{F/F^+}$. In particular, since all the finite places $v \in S$ split in $F/F^+$, $\rho(\bar r)|_v$ and all its local deformations are valued in the group $\GL{N}$. Thus we will omit $\delta_{F/F^+}$ from our notation when we discuss about local deformations.

Recall from the discussion in \S \ref{sec:notation} that for each $v\in S$ we can consider the lifting rings $R^{\Box}_{\bar r|_v}$ and 
$R^{\Box}_{\rho(\bar{r})|_v}$.
For $v \in S$ not above $p$, choose an irreducible component $\mc{C}(r_v)$, resp.  
$\mc{C}(\rho(r_v))$
of $R^{\Box}_{\bar{r}|_v} \otimes \bQp$, resp. 
$R^{\Box}_{\rho(\bar{r})|_v} \otimes \bQp$ containing $r_v$, resp.
$\rho(r_v)$.
such that under the natural map 
\[ 
\op{Spec} R^{\Box}_{\bar r|_v} \otimes \bQp \to 
\op{Spec} R^{\Box}_{\rho(\bar{r})|_v} \otimes \bQp,
\]
$\mc C(r_v)$ maps to
$\mc{C}(\rho(r_v))$. 

Similarly, for $v \vert p$, the fixed inertial type and fixed $p$-adic Hodge type of $r_v$ induces corresponding data for $\rho(r_v)$. Choose $\mc{C}(r_v)$ to be a potentially crystalline component mapping into a potentially crystalline component $\mc{C}(\rho(r_v))$ containing the potentially diagonalizable point $\rho(r_v)$ under the natural map
\[
\op{Spec} R^{\Box, \un{\mu}(r_v), \sigma_v}_{\bar r|_v}\otimes \bQp \to 
\op{Spec} R^{\Box, \un{\mu}(\rho(r_v)), \rho(\sigma_v)}_{\rho(\bar{r})|_v}\otimes \bQp.
\]

We now the global deformation rings, for $\bar{r}$ and $\rho(\bar{r})$, by considering lifts that locally lie on the irreducible components we have just specified. More precisely, following the formalism of \cite[\S 4.2]{bg:Gdef}, we let $R^{\mr{univ}}_{G}$ be the quotient of the universal, unramified outside $S$ deformation ring for $\bar{r}$ corresponding to the fixed set of components $\{\mc{C}(r_v)\}_{v \in S}$. We similarly define $R^{\mr{univ}}_{\mr{GL}}$ corresponding to the local components $\{\mc{C}(\rho(r_v))\}$ (and fixed polarization $\delta_{F/F^+}$). These rings all exist by absolute irreducibility of the respective residual representations, and by the discussion in \cite[\S 4.2]{bg:Gdef} (\cite[Lemma 3.4.1]{bg:Gdef} plays a key role here).
By construction, there is a natural $\mc O$--algebra map
\[
R^{\mr{univ}}_{\mr{GL}} \to R^{\mr{univ}}_{\mr{G}}.
\]

\begin{lem}\label{defcompare}
The map $R^{\mr{univ}}_{\mr{GL}} \to R^{\mr{univ}}_{G}$ is surjective.
\end{lem}
\begin{proof}
The tangent space of $R^{\mr{univ}}_{\mr{GL}}$ is a subspace of $H^1(\Gamma_{F^+,S}, \rho(\bar{r})(\mf{gl}_{N}))$ and
the tangent space of $R^{\mr{univ}}_{G}$ is a subspace of $H^1(\Gamma_{F^+,S}, \bar r(\mf{g}))$. For $p$ large enough, $\mf{gl}_{N}$ is by \cite[Proposition 2]{ser:ss} a semisimple $G$--module, so \textit{a fortiori} $\bar r(\mf{g})$ is a $\Gamma_{F^+, S}$--direct summand of $\rho(\bar{r})(\mf{gl}_{N})$. It follows that the natural map $H^1(\Gamma_{F^+,S}, \bar r(\mf{g})) \to H^1(\Gamma_{F^+,S}, \rho(\bar{r})(\mf{gl}_{N}))$ is injective; the dual map is surjective, and we conclude by Nakayama's lemma that the map on universal deformation rings without local conditions is surjective. It then follows immediately that $R^{\mr{univ}}_{\mr{GL}} \to R^{\mr{univ}}_{G}$ is surjective as well. 
\end{proof}

By the proof of \cite[Theorem 4.3.1]{blggt}, $R^{\mr{univ}}_{\mr{GL}}$ is $\mc O$--finite.
Lemma \ref{defcompare} then implies that $R^{\mr{univ}}_{G}$ is $\mc O$--finite. We claim that $R^{\mr{univ}}_{G}$ has Krull dimension at least one. Indeed, this follows from \cite[Theorem B]{bg:Gdef}: the assumptions there are satisfied since we have assumed that $\bar r$ is odd, $R \circ \bar r|_{F^+(\mu_p)}$ is irreducible (which, under our assumption on $p$, implies that $H^0(\Gamma_{F^+},\bar r(\mf{g})(1))$ vanishes), and the Hodge--Tate cocharacters are regular. Thus $R^{\mr{univ}}_{G}$ has a $\bQp$--point, say $r: \Gamma_{F^+} \to G(\bQp)$. Then $R \circ r$ satisfies the hypotheses in \cite[Theorem C]{blggt}: (1) is clear; (2) holds since $R$ is valued in $\SO{N}$ by assumption; (3) holds because for $v \vert p$, $R \circ r|_v$ lies on the same (potentially crystalline) component as the potentially diagonalizable point $R\circ r_v=\rho(r_v)$; since potentially crystalline deformation rings are regular, there is a unique potentially crystalline component of $\op{Spec} R^{\Box, \un{\mu}(\rho(r_v)), \rho(\sigma_v)}_{\rho(\bar{r})|_v}\otimes \bQp$ passing through $\rho(r_v)$. It follows that $R \circ r|_v$ is potentially diagonalizable. (4) holds by our assumption on $\bar r$. Therefore, Theorem C of \textit{loc.cit.} implies that $R \circ r$ is potentially automorphic and belongs to a strictly compatible system of $\ell$-adic representations. Let us denote this compatible system by $\{R_{\lambda}\}$, where $R_{\lambda}: \Gamma_{F^+} \to \gl{N}{\ov{M}_{\lambda}}$ is a continuous representation with $M$ a number field and $\lambda$ primes of $M$. Suppose that for $v \vert p$, the Hodge--Tate weights of $R \circ r_v$ (and hence $R \circ r|_v$) are extremely regular in the sense of \cite[\S 2.1]{blggt}, then \cite[Theorem D]{blggt} (which ultimately relies on Larsen's work \cite{lar:max}) and the standard Brauer induction argument (see \cite[Theorem 5.5.1]{blggt}) imply that for a density one set of rational primes $l$, $R_{\lambda}$ is irreducible for $\lambda \vert l$.

We summarize the above discussion in the following:
\begin{thm}\label{lifting thm}
Retain the assumptions on $G$ and $\bar r: \Gamma_{F^+,S} \to G(\bFp)$ imposed at the beginning of this section. Then $\bar r$ has a lift $r: \Gamma_{F^+} \to G(\bQp)$ unramified outside $S$ such that 
\begin{itemize}
    \item For each $v\in S$ not lying above $p$, $r|_v:= r|_{\Gamma_{F^+_v}}$ and $r_v$ lie on the same irreducible component of $\op{Spec}(R^{\Box}_{\bar r|_v} \otimes \bQp)$.
    \item For each $v \vert p$, $r|_v$ and $r_v$ lie on the same irreducible component of $\op{Spec} R^{\Box, \un{\mu}(r_v), \sigma_v}_{\bar r|_v}\otimes \bQp$. Moreover, $R \circ r|_v$ is potentially diagonalizable in the sense of \cite{blggt} with regular Hodge--Tate weights.  
    \item $R \circ r$ is potentially automorphic in the sense of \cite{blggt}.
    \item $R \circ r$ is part of a strictly compatible system $R_{\lambda}: \Gamma_{F^+} \to \gl{N}{\ov{M}_{\lambda}}$ indexed by finite places $\lambda$ of some number field $M$. If in addition, the Hodge--Tate weights of $R \circ r|_v$  are extremely regular in the sense of \cite[\S 2.1]{blggt}, then for a density one set of rational primes $l$, $R_{\lambda}$ is irreducible for $\lambda \vert l$.
\end{itemize}
\end{thm}

\section{Compatible systems of Spin Galois representations}\label{sec:compatible systems}
In this section, we use the lifting method of \S \ref{sec:lifting gal} to deform the representation $\bar r: \galQ \to \spin{m}{\Fp}$ in Proposition \ref{mod p rep} to a characteristic zero representation satisfying favorable local conditions. Moreover, we will show that the associated compatible system of Galois representations $\{R_{\lambda}\}$ has $\Spin{m}$ monodromy for a density one set of rational primes $l$ below $\lambda$. 

\textit{For the rest of this paper, we assume that $m$ satisfies the parity assumption in \S \ref{sec:mod p rep}}.

We need to specify the local deformations of $\bar r$, following the notation in \S \ref{sec:lifting gal}.
Let $S$ be the union of $\infty$, $p, q$ and the places of $\rats$ at which $\bar r$ is ramified. We choose a quadratic field $F$ such that $F$ does not contain $\zeta_p$, $\mr{spin} \circ \bar r|_{F(\mu_p)}$ remains irreducible, and all the finite places in $S$ split in $F/\rats$. Suppose that we have a homomorphism $r: \galQ \to \spin{m}{A}$ for some ring $A$; as in \S \ref{sec:lifting gal}, we define $\rho(r) \colon \Gamma_{\rats} \to \mc{G}_{N}(A)$ as the composite (recall that the spin representation lands in $\SO{N}$ for $m$ satisfying the parity assumption)
\[
\Gamma_{\rats} \xrightarrow{\mr{spin}(r) \times \mr{res}_{F}} \mr{SO}_{N}(A) \times \mr{Gal}(F/\rats) \to \mc{G}_{N}(A),
\] 
where the last map sends $g\in \mr{SO}_{N}(A)$ to its image in $\gl{N}{A}$ and sends the nontrivial element of $\gal{F}{\rats}$ to $\jmath$.

\textit{Deformation condition at $p$}.
Recall that $\bar r|_{p}=\prod_{\alpha \in \Delta} \alpha^{\vee} \circ \bkp[l]^{n_{\alpha}}$. Fix 
\[
\chi_T=\prod_{\alpha \in \Delta} \alpha^{\vee} \circ (\kappa^{\tilde n_{\alpha}} \cdot [\bkp]^{-n_{\alpha}} \cdot [\bkp[l]]^{n_{\alpha}})
\]
lifting $\bar r|_{p}$, where $[\bkp]$, resp. $[\bkp[l]]$ denotes the Teichmuller lift of $\bkp$, resp. $[\bkp[l]]$, and $\tilde n_{\alpha} \equiv n_{\alpha} \mod l-1$ are greater than one and sufficiently general ($\tilde n_{\alpha}>1$ ensures that our characteristic zero lifts are potentially crystalline, cf. \cite[Lemma 4.8]{pat:exm}). 
We take $r_p$ to be $\chi_T$. Let $\mc C(r_p)$ be an ordinary potentially crystalline component of $\op{Spec} R^{\Box, \mu(r_p), \sigma(r_p)}_{\bar r|_p}\otimes \bQp$ containing $r_p$ determined by the ordinary deformation condition of \cite[Definition 4.1]{pat:exm} with respect to the Borel $B \supset T$ associated to $\Delta$ (Proposition \ref{mod p rep}, (5) ensures that this deformation condition is well-defined). Note that the genericity of $\{\tilde n_{\alpha}\}$ for $\alpha \in \Delta$ and the fact that $\Lambda_{spin}$ is multiplicity-free imply that $\mr{spin} \circ r_p$ has regular Hodge--Tate weights. We may and do assume that the integers $\tilde n_{\alpha}$ are chosen so that the Hodge--Tate weights of $\mr{spin} \circ r_p$ are extremely regular in the sense of \cite[\S 2.1]{blggt}. Let $\mc C(\rho(r_p))$ ($\rho(r_p)=\mr{spin} \circ r_p$) be the corresponding potentially crystalline component of $\op{Spec} R^{\Box, \mu(\rho(r_p)), \rho(\sigma(r_p))}_{\rho(\bar{r})|_p}\otimes \bQp$.

\textit{Deformation condition at $q$}.
Define $\rho_q: \Gamma_{\rats_q} \to \sl{2}{\Zp}$ by $\rho_q(\sigma_q)=\bp q&0\\0&1 \ep$ and $\rho_q(\tau_q)=\bp 1&p\\0&1 \ep$, where $\tau_q$ is a topological generator of $\gal{\rats_q^{tame}}{\rats_q^{unr}}$ and $\sigma_q \in \gal{\rats_q^{tame}}{\rats_q}$ is a lift of the Frobenius element in $\gal{\rats_q^{unr}}{\rats_q}$ satisfying $\sigma_q \tau_q \sigma_q^{-1}=\tau_q^q$. We set $r_q=\varphi \circ \rho_q$, where $\varphi: \SL{2} \to \Spin{m}$ is the principal $\SL{2}$ homomorphism with respect to $T$ and a Borel $B \supset T$ of $\Spin{m}$. Then $r_q$ lifts $\bar r|_q$. Suppose that $l$ is greater than $h_{\Spin{m}}$ (see Proposition \ref{mod p rep}), then the Steinberg deformation condition in \cite[\S 4.3]{pat:exm} (with respect to $B$) is well-defined, which corresponds to an irreducible component $\mc C(r_q)$ of $\op{Spec} R^{\Box}_{\bar r|_q} \otimes \bQp$ passing through $r_q$ . Let $\mc C(\rho(r_q))$ be the corresponding irreducible component of $\op{Spec} R^{\Box}_{\rho(\bar r)|_q} \otimes \bQp$.

\textit{Deformation conditions at other primes}.
Let $v$ be a prime of $\rats$ outside $\{p,q\}$ at which $\bar r$ is ramified. We have $\bar r(I_v) \subset T_{\Fp}\cdot \tilde D$, which implies (since we are assuming $p$ is sufficiently large) that $|\bar r(I_v)|$ is not divisible by $p$. We take the minimal prime to $p$ deformation condition of \cite[\S 4.4]{pat:exm} at $v$, which determines irreducible components of the local deformation rings as before. 

\begin{prop}\label{char zero spin rep}
Let $\bar r: \galQ \to \spin{m}{\bFp}$ be the mod $p$ representation in Proposition \ref{mod p rep} (recall that $m \geq 7$ and $m \equiv 0,1,7 \mod 8$). Then $\bar r$ admits a characteristic zero lift $r: \galQ \to \spin{m}{\bZp}$ unramified outside a finite set of primes such that 
\begin{enumerate}
    \item $r|_p$ is ordinary in the sense of \cite[\S 4.1]{pat:exm} with Hodge--Tate cocharacter $\prod_{\alpha \in \Delta} (\alpha^{\vee})^{\tilde n_{\alpha}}$ and $R \circ r|_p$ is potentially diagonalizable in the sense of \cite{blggt} with extremely regular Hodge--Tate weights (in the sense of \cite[\S 2.1]{blggt}).  
    \item $r|_q$ is Steinberg in the sense of \cite[\S 4.3]{pat:exm}, which is moreover ramified.
    \item The Zariski closure of the image of $r$ equals $\Spin{m}$.
    \item $R \circ r$ is potentially automorphic, i.e. there exist a totally real extension $F^+/\rats$ and a regular L--algebraic, cuspidal automorphic representation $\Pi$ of $\gl{N}{\mb A_{F^+}}$ such that 
$\mr{spin} \circ r|_{\Gal{F^+}} \cong r_{\Pi,\iota}$, where $\iota: \bQp \xrightarrow{\sim} \cmplx$ is a fixed field isomorphism.
\end{enumerate}
\end{prop}
\begin{proof}
By Proposition \ref{mod p rep}, $\bar r$ satisfies the hypotheses at the beginning of \S \ref{sec:lifting gal} with $F^+=\rats$, $G=\Spin{m}$ and $R=\mr{spin}$. (1) and (4) then follow immediately from Theorem \ref{lifting thm}. By the same theorem, $r|_q$ lies on $\mc C(r_q)$, and hence it is Steinberg. We will show that $r|_q$ is ramified. Since $\Pi$ is cuspidal, $\Pi_w$ is generic for all places $w$ of $F^+$, and therefore local--global compatibility (\cite{car:locgl}) and \cite[Lemma 1.3.2]{blggt} imply that $\mr{spin} \circ r|_{\Gal{F^+_w}}$ is a smooth point on its (generic fiber) local lifting ring. Let $w|q$ be any place of $F^+$ above $q$. Since $r|_{\Gal{F^+_w}}$ lies on the same irreducible component as the restriction $r_q|_{\Gal{F^+_w}}$ of the Steinberg-type lift constructed above, so do $\mr{spin} \circ r|_{\Gal{F^+_w}}$ and $\mr{spin} \circ r_q|_{\Gal{F^+_w}}$. The latter is visibly a smooth point, so by \cite[Lemma 1.3.4 (2)]{blggt} (due to Choi) the inertial restrictions $\mr{spin} \circ r|_{I_{F^+_w}}$ and $\mr{spin} \circ r_q|_{I_{F^+_w}}$ are isomorphic. In particular, $r|_q$ is ramified, which proves (2). It remains to prove (3).
Let $G_r$ be the Zariski closure of $r$, which is reductive since $r$ is irreducible. Moreover, (2) implies that $G_r$ contains a regular unipotent element of $\Spin{m}$. It then follows from a theorem of Dynkin (\cite[Theorem A]{ss97}) that the Lie algebra of $G_r$ is one of $\mf{sl}_2$, $\mf g_2$ and $\mf{so}_{2n+1}$ when $m=2n+1$ (where $\mf g_2$ may occur only when $m=7$); one of $\mf{sl}_2$, $\mf g_2$, $\mf{so}_{2n-1}$ and $\mf{so}_{2n}$ when $m=2n$ (where $\mf g_2$ may occur only when $m=8$). 
The argument of \cite[Lemma 7.8]{pat:exm} then implies that $G_r=\Spin{m}$ if the integers $\{\tilde n_{\alpha}\}_{\alpha \in \Delta}$ are taken to be sufficiently general.
\end{proof}

\begin{thm}\label{spin compatible system} 
Let $r$ be as in Proposition \ref{char zero spin rep}. Then $\mr{spin} \circ r$ is part of a strictly compatible system $R_{\lambda}: \Gal{\rats} \to \gl{N}{\ov{M}_{\lambda}}$ indexed by finite places $\lambda$ of some number field $M$ with regular Hodge--Tate weights such that for a density one set of rational primes $l$, the Zariski closure of the image of $R_{\lambda}$ equals $\Spin{m}$ for $\lambda \vert l$.
\end{thm}
\begin{proof}
By Proposition \ref{char zero spin rep}, (1) and Theorem \ref{lifting thm}, (4), for a density one set $\Sigma$ of rational primes $l$, $R_{\lambda}$ is irreducible for $\lambda \vert l$. Then \cite[Theorem 4]{lp:invdim} implies that for $l \in \Sigma$ and $\lambda|l$, $R_{\lambda}$ factors through $\Spin{m}$. In fact, let $G_{\lambda}$ be the Zariski closure of the image of $R_{\lambda}$. By Proposition \ref{char zero spin rep}, (3), $G_{\lambda}=\Spin{m}$ for some $\lambda|p$, so $(G_{\lambda}, \mr{std}_{\lambda})$ (where $\mr{std}_{\lambda}$ denotes the canonical inclusion $G_{\lambda} \subset \GL{N}$) is similar to $(\mr{Spin}_{m}, \op{spin})$ in the sense of \cite{lp:invdim}, i.e. they have the same formal characters (since the group of connected components and the formal character of the connected component of the monodromy group are independent of $l$; see \cite[Propositions 6.12 and 6.14]{lp:lind}). If $m$ is even, then \cite[Theorem 4]{lp:invdim} implies that $G_{\lambda}=\Spin{m}$; if $m=2n+1$ is odd,  
the only `basic similarity classes' in the sense of \textit{loc.cit.} relevant to $(\mr{Spin}_{2n+1}, \op{spin})$ are products $(\prod_{i=1}^k \mr{Spin}_{2n_i+1}, \boxtimes_i \op{spin}_{2n_i+1})$ for some decomposition $n_1+ \cdots +n_k= n$, where $\op{spin}_{2n_i+1}$ denotes the appropriate spin representation. This external product is simply the restriction of the $2^n$-dimensional spin representation to $\prod \mr{Spin}_{2n_i+1} \subset \mr{Spin}_{2n+1}$, so in any case $G_{\lambda} \subset \mr{GL}_{2^n}$ factors (up to conjugacy) through $\mr{Spin}_{2n+1} \subset \mr{GL}_{2^n}$. It follows that $R_{\lambda}$ factors as $\op{spin}(r_{\lambda})$ for some $r_{\lambda} \colon \Gamma_{\rats} \to \mr{Spin}_{m}(\ov{M}_{\lambda})$. It remains to show that when $m$ is odd, $G_{\lambda}=\Spin{m}$ for $\lambda \vert l$ with $l$ belonging to a density one set of rational primes. By Proposition \ref{char zero spin rep} (2) and (4) and local--global compatibility (\cite{car:locgl}), $G_{\lambda}$ contains (up to conjugacy) the image under the spin representation of a regular unipotent element of $\Spin{m}$ as long as $\lambda$ is not above $q$. By \cite[Lemma 3.5]{bce+}, $r_{\lambda}$ (for $\lambda$ above some $l \in \Sigma-\{q\}$) then has image containing a regular unipotent element. We deduce from Dynkin's theorem (\cite[Theorem A]{ss97}) that the Zariski closure of the image of $r_{\lambda}$ equals $\Spin{m}$.
\end{proof}

Using some results of Larsen (\cite{lar:max}), we give an application of the above to the inverse Galois problem for the $\Fp$--points of spin groups.
\begin{cor}\label{igp for spin}
For $m \geq 7$, $m \equiv 0,1,7 \mod 8$, $\spin{m}{\Fp}$ is the Galois group of a finite Galois extension of $\rats$ for $p$ belonging to a set of rational primes of positive density.
\end{cor}
\begin{proof}
Let $\{R_{\lambda}\}$ be as in Theorem \ref{spin compatible system}. By \cite[Lemma 5.3.1,(3)]{blggt}, after replacing $M$ by a finite extension if necessary, we may assume that $R_{\lambda}$ is valued in $\gl{N}{M_{\lambda}}$. Let $\mr{cs}(M/\rats)$ be the set of rational primes $l$ that split completely in $M/\rats$. Then \cite[Theorem 3.17]{lar:max} implies that there is a subset $T $ of $\mr{cs}(M/\rats)$ of relative density one, such that for $l\in T$ and $\lambda \vert l$, $R_{\lambda}(\galQ)$ is a hyperspecial maximal compact subgroup of $G_{\lambda}(\Ql)$. By Theorem \ref{spin compatible system}, $G_{\lambda}(\Ql)$ equals $\spin{m}{\Ql}$ for a density one set of rational primes $l$. It follows that $R_{\lambda}(\galQ)=\spin{m}{\Zl} \rsurj \spin{m}{\Fl}$ for $l$ belonging to a set of rational primes of positive density.
\end{proof}

\begin{rmk}\label{zywina}
Using a completely analogous argument, one can show that for $n \geq 2$, $\so{2n+1}{\Fp}$, resp. $\so{4n}{\Fp}$ is the Galois group of a finite Galois extension of $\rats$ for $p$ belonging to a set of rational primes of positive density. We compare this with the main result in \cite{zyw:igp}, which shows that the finite simple groups $\Omega_{2n+1}(p)$ and $\mr{P\Omega}_{4n}^+(p)$ both occur as the Galois group of a Galois extension of $\rats$ for all integers $n \geq 2$ and all primes $p \geq 5$.  
\end{rmk}

\section{Realization of $R_{\lambda}$ in the cohomology of algebraic varieties}\label{sec:coh}
\begin{thm}\label{realization in coh}
The compatible system of Galois representations $\{R_{\lambda}\}$ in Theorem \ref{spin compatible system} is motivic in the following sense: there is a smooth projective variety $X/\rats$ and integers $i$ and $j$ such that $R_{\lambda}$ is a $\galQ$--subrepresentation of $\coh{i}{X_{\bQ}}{\bQl(j)}$. 
\end{thm}
The proof of this theorem is almost identical to the argument in \cite[\S 5]{bce+}, except for a few minor modifications. However, for the reader's convenience, we reproduce here part of their argument. The main reference for this argument is \cite{shi:shi}. In \cite[\S 5]{bce+}, the dimension of $\{R_{\lambda}\}$ is odd, which falls into (Case ST) of \cite{shi:shi}; in our case, the dimension of $\{R_{\lambda}\}$ (being a power of 2) is even, which falls into (Case END) of \cite{shi:shi}. This is why we need to modify the argument of \cite[\S 5]{bce+} to prove Theorem \ref{realization in coh}. Fix an isomorphism $\iota_l: \bQl \xrightarrow{\sim} \cmplx$; it is implicit in all of the constructions of \cite{shi:shi}. The compatible system $\{R_{\lambda}\}$ is by construction potentially automorphic, i.e. there is a totally real field $F^+/\rats$ such that $R_{\lambda}|_{\Gal{F^+}} \cong r_{l,\iota_l}(\Pi^0)$ for some cuspidal automorphic representation $\Pi^0$ of $\gl{N}{\mb A_{F^+}}$ (in the notation of \cite[Theorem 2.1.1]{blggt}). Let $F/F^+$ be a quadratic CM extension. Replacing $\Pi^0$ by its base change to $F$, we may assume that $R_{\lambda}|_{\Gal{F}} \cong r_{l,\iota_l}(\Pi^0)$ for a conjugate self-dual cuspidal automorphic representation $\Pi^0$ of $\gl{N}{\mb A_{F}}$. Then $R_{\lambda}|_{\Gal{F}} \cong R_l(\Pi^{0,\vee})$ in the notation of \cite[Theorem 7.5]{shi:shi} (note that the normalizations for the local Langlands correspondence used in \cite{blggt} and \cite{shi:shi} differ by a dual: see \cite[Theorem 2.1.1]{blggt} and \cite[\S 2.3]{shi:shi}). Replacing $\Pi^0$ by its dual, we may write
\[R_{\lambda}|_{\Gal{F}} \cong R_l(\Pi^0).\] 
We may further adjust the integers $\{\tilde n_{\alpha}\}$ in Proposition \ref{char zero spin rep} so that the associated automorphic representation $\Pi^0$ is \textit{slightly regular} (a.k.a. Shin regular) in the sense of \cite{shi:shi} (this is necessary in Case END). 

Let $n=N+1$ ($N$ is `$m$' in \cite{shi:shi}). \footnote{To keep our notation consistent with that of \cite{shi:shi}, we borrow the same letter $n$ from the preceding sections which was used to denote the rank of the spin group. We hope this does not cause any confusion.} We recall the construction of $R_l(\Pi^0)$. Let $E$ be an imaginary quadratic field not contained in $F$ satisfying the four bulleted conditions in Step (II) of the proof of \cite[Theorem 7.5]{shi:shi}. Replace $F$ by $FE$ and $\Pi^0$ by $\mr{BC}_{FE/F}(\Pi^0)$. Then the triple $(E,F,\Pi^0)$ satsfies the six bulleted conditions at the beginning of Step (I) of the proof of \textit{loc.cit}. Let $F'$ be an imaginary quadratic extension of $F^+$ satisfying the the three bulleted conditions in Step (I) of \textit{loc.cit}. Then replace $F$ by $FF'$ and $\Pi^0$ by $\mr{BC}_{FF'/F}(\Pi^0)$ so that \cite[Proposition 7.4]{shi:shi} applies to the triple $(E,F,\Pi^0)$. There is a Hecke character $\psi$ of $\mb A_E^{\times}/E^{\times}$ such that, setting $\Pi=\psi \boxtimes \Pi^1$ (which is an automorphic representation of the group $\gl{1}{\mb A_E} \times \gl{n}{\mb A_F}$; $\Pi^1$ is constructed from $\Pi^0$ in \cite[\S 7.1]{shi:shi}), we have (in the notation of \cite[Corollary 6.10]{shi:shi}
\footnote{In our case, $i=1$, $m_i=n-1$.}
)
\[
R_l(\Pi^0):=R''_l(\Pi):=R'_l(\Pi) \otimes \mr{rec}_{l,\iota_l} 
\big( (\varpi \circ N_{F/E}) \otimes |.|^{1/2} \big), 
\footnote{$\varpi: \mb A_E^{\times}/E^{\times} \to \cmplx^{\times}$ is the Hecke character defined in \cite[\S 3.1]{shi:shi}.}
\]
where $R'_l(\Pi):=\tilde R'_l(\Pi) \otimes \mr{rec}_{l,\iota_l}(\psi^c)|_{\Gal{F}}$ and
\[ 
C_G \cdot \tilde R'_l(\Pi)= \sum_{\pi^{\infty} \in \mc R_l(\Pi)} R_{\xi,l}^{n-1}(\pi^{\infty})^{ss}
\]
(see \cite[(5.5),(6.23)]{shi:shi}; $G$ the unitary similitude group defined in \cite[\S 5.1]{shi:shi}).
The rest of the argument is identical to that of \cite[\S 5]{bce+}: one shows that $\tilde R'_l(\Pi)$ is a subrepresentation of the cohomology of a smooth projective variety over $F$, which implies that $R_{\lambda}|_{\Gal{F}}$ (which is a twist of $\tilde R'_l(\Pi)$) is a subrepresentation of a smooth projective variety over $F$ after possibly replacing $F$ by a finite extension (using \cite[IV. Proposition D.1]{dmos82}). It then follows from Frobenius reciprocity and the irreducibility of $R_{\lambda}$ that $R_{\lambda}$ is a subrepresentation of a smooth projective variety over $\rats$. This finishes the proof of Theorem \ref{realization in coh}.

\section{Comparison with the work of Kret and Shin}\label{sec:KS}
In this section, we explain how the main theorem of \cite{kret-shin} implies a stronger version of Theorem \ref{main theorem} in the case when $m=2n+1$ is odd. However, \cite{kret-shin} does not have any implications on the case when $m$ is even, so Theorem \ref{main theorem} contains the first sighting of the group $\Spin{m}$ for $m \equiv 0 \mod 8$ as any sort of monodromy group of motivic Galois representations.

The maximal torus $T$ of $\Spin{2n+1}$ may be described as 
\[T \xrightarrow{\sim} \{ (z,t_1, ... ,t_n) \in \mb G_m^{n+1}: z^2=t_1\cdots t_n \}.\]
For $1 \leq i \leq n$, let $\chi_i \in X^{\bullet}(T)$ be the character defined by $\chi_i(z,t_1, ... ,t_n)=t_i$, and let $\lambda_i \in \frac{1}{2}X_{\bullet}(T)$ be the cocharacter defined by $(2\lambda_i)(t)=(t,1, ... ,1,t^2,1, ... ,1)$, where $t^2$ is located at the $(i+1)$-th entry. Then the character lattice $X^{\bullet}(T)$ is generated by $\chi_1, ... ,\chi_n$ and $\frac{1}{2}(\sum_i \chi_i)$, where $\frac{1}{2}(\sum_i \chi_i)$ sends $(z,t_1, ... ,t_n)$ to $z$; and the cocharacter lattice $X_{\bullet}(T)$ equals $\{\sum_i m_i\lambda_i: m_i \in \ints, \sum m_i \in 2\ints \}$. 

Note that $\Sp{2n}^{\vee}=\SO{2n+1}$, $\mr{PSp}_{2n}^{\vee}=\Spin{2n+1}$ and $\GSp{2n}^{\vee}=\GSpin{2n+1}$.

Let $\pi_{\infty}$ be a discrete series representation of $\sp{2n}{\reals}$ with a regular and tempered L--parameter $\phi_{\infty}: W_{\reals} \to \so{2n+1}{\cmplx}$ defined by 
\[ \phi_{\infty}(z)=diag(z^{m_1}\bar z^{-m_1}, ... ,z^{m_n}\bar z^{-m_n},1,z^{-m_n}\bar z^{m_n}, ... ,z^{-m_1}\bar z^{m_1}) \] for $z \in \cmplx^{\times}$, $m_i \in \ints$ and
$\phi_{\infty}(j)=J$, where $J=[a_{ij}]$ is the $(2n+1)$ by $(2n+1)$ matrix with $a_{n+1,n+1}=(-1)^n$, $a_{i,2n-i+2}=1$ for $i \neq n+1$, and $a_{ij}=0$ otherwise. 
We want to lift $\phi_{\infty}$ to $\spin{2n+1}{\cmplx}$. Note that $\phi_{\infty}|_{\cmplx^{\times}}$ has a unique lift 
\[ z \mapsto (\sum 2m_i\lambda_i)(\frac{z}{|z|}). \]
In particular, $-1 \mapsto (-1)^{\sum m_i}$ (we use $-1$ to also denote the nontrivial element in the center of $\Spin{2n+1}$ by abuse of notation). The element $j$ necessarily maps to a lift of $J \in \so{2n+1}{\cmplx}$ in $\spin{2n+1}{\cmplx}$, which is either $\omega_1\cdots\omega_n$ or $-\omega_1\cdots\omega_n$ (in the notation of Lemma \ref{NT sequence for Spin} and the paragraph preceding it). We have $(\omega_1\cdots\omega_n)^2=(-1)^{n(n+1)/2}$. To have a well-defined homomorphism $\tilde\phi_{\infty}: W_{\reals} \to \spin{2n+1}{\cmplx}$ lifting $\phi_{\infty}$, we must ensure that $\tilde\phi_{\infty}(j)\tilde\phi_{\infty}(z)\tilde\phi_{\infty}(j)^{-1}=\tilde\phi_{\infty}(\bar z)$ and $\tilde\phi_{\infty}(j)^2=\tilde\phi_{\infty}(-1)$. The first equality is automatic. The second equality is by the above equivalent to $(-1)^{n(n+1)/2}=(-1)^{\sum m_i}$. Therefore, $\phi_{\infty}$ lifts to $\spin{2n+1}{\cmplx}$ if and only if $\sum m_i$ has the same parity as $n(n+1)/2$. On the automorphic side, this says $\pi_{\infty}$ descends to a discrete series representation of $\mr{PSp}_{2n}(\reals)$ if and only if the integers $m_i$ appearing in its L--parameter satisfies $\sum m_i\equiv \frac{n(n+1)}{2} \mod 2$.

Now we ask: Of the discrete series that descends to $\mr{PSp}_{2n}(\reals)$, are any of them L--algebraic in the sense of \cite{buzzard-gee}? Consider their L--parameter $\tilde\phi_{\infty}|_{\cmplx}$, $z \mapsto (\sum 2m_i\lambda_i)(\frac{z}{|z|})=z^{\sum m_i\lambda_i}\bar z^{-\sum m_i\lambda_i}$, we have $\sum m_i\lambda_i \in X_{\bullet}(T)$ if and only if $\sum m_i \in 2\ints$. Therefore we have 

\begin{lem}\label{L-parameters}
Let $\pi_{\infty}$ and $\phi_{\infty}$ be as before. Then 
\begin{enumerate}
    \item $\pi_{\infty}$ descends to a discrete series representation of $\mr{PSp}_{2n}(\reals)$ if and only if $\sum m_i\equiv \frac{n(n+1)}{2} \mod 2$.
    \item $\pi_{\infty}$ descends to an L--algebraic discrete series representation of $\mr{PSp}_{2n}(\reals)$ if and only if $n \equiv 0,3 \mod 4$ and $\sum m_i \in 2\ints$.
\end{enumerate}
\end{lem}

A trace formula argument shows that there are infinitely many cuspidal automorphic representations $\pi$ of $\mr{PSp}_{2n}(\mb A_{\rats})$ such that its archimedean component $\pi_{\infty}$ satisfies Lemma \ref{L-parameters}, (2) and $\pi$ is Steinberg at a finite place of $\rats$ (cf. \cite{clo:lmf}). Let $\tilde\pi$ be the cuspidal automorphic representation of $\gsp{2n}{\mb A_{\rats}}$ associated to $\pi$ via the natural map $\mr{pr}: \GSp{2n} \rsurj \mr{PSp}_{2n}$. We want to use the main theorem of \cite{kret-shin} to associate to $\tilde \pi$ $\GSpin{2n+1}$--valued Galois representations. 

\begin{lem}\label{L-cohomological}
The infinity component $\tilde\pi_{\infty}$ of $\tilde\pi$ is L--cohomological in the sense of \cite{kret-shin}, i.e. $\tilde{\pi}_{\infty} \cdot |\lambda_0|^{\frac{n(n+1)}{4}}$ is cohomological in the sense of \cite[Definition 1.12]{kret-shin} ($\lambda_0$ corresponds to the similitude character of $\mr{GSp}_{2n}$).
\end{lem}
\begin{proof}
Since $\pi_{\infty}$ is a discrete series representation of $\sp{2n}{\reals}$ with infinitesimal character $\mu(\pi_{\infty})$ (where $\mu(\pi_{\infty})$ is determined by writing $\phi_{\infty}(z)=z^{\mu(\pi_{\infty})}\bar z^{\nu(\pi_{\infty})}$), it is $\xi^\vee_{\mu(\pi_{\infty})-\rho}$--cohomological (\cite[Theorem V.3.3]{borel-wallach}), where $\xi_{\mu(\pi_{\infty})-\rho}$ is the highest-weight $\mu(\pi_{\infty})-\rho$ representation of $\mr{Sp}_{2n}(\cmplx)$, and $\rho$ is the half-sum of positive roots for the choice of root basis used also to parametrize highest weights. We need an extension $\tilde{\xi}$ of $\xi_{\mu(\pi_{\infty})-\rho}$ to $\mr{GSp}_{2n}(\cmplx)$ such that (letting $K$ be a maximal compact subgroup of $\mr{Sp}_{2n}(\reals)$)
\[
H^*(\mf{gsp}_{2n}(\cmplx), K\cdot Z_{\mr{GSp}_{2n}}(\reals); \tilde{\pi}_{\infty}\cdot|\lambda_0|^{\frac{n(n+1)}{4}} \otimes \tilde{\xi}^\vee) \neq 0.
\]
From the definition of this cohomology group, and the fact that $\pi_{\infty} \otimes \xi_{\mu(\pi_{\infty})-\rho}^\vee$ has non-zero $(\mf{sp}_{2n}(\cmplx), K)$--cohomology, this reduces to checking that we can find an extension $\tilde{\xi}^\vee$ of $\xi^\vee_{\mu(\pi_{\infty})-\rho}$ such that $\tilde{\xi}$ and $\tilde{\pi}_{\infty}\cdot|\lambda_0|^{\frac{n(n+1)}{4}}$ have the same central character; note that we will only have to check this on $\reals^\times_{>0} \subset Z_{\mr{GSp}_{2n}}(\reals)$ since $-1 \in K$. We choose $\tilde{\xi}$ to be the representation with highest weight $\mu(\tilde{\pi}_{\infty})-\rho+\frac{n(n+1)}{4}\lambda_0 \in X^{\bullet}(T_{\GSp{2n}})$, where $\mu(\tilde{\pi}_{\infty})=\sum m_i\lambda_i$, which (since $\sum_i m_i$ is even) lies in $X_{\bullet}(T_{\Spin{2n+1}})
=X^{\bullet}(T_{\mr{PSp}_{2n}}) \subset X^{\bullet}(T_{\GSp{2n}})$, and both $\rho$ and $\frac{n(n+1)}{4}\lambda_0$ are integral since $n \equiv 0,3 \mod 4$. The lemma is now proven, since this highest weight representation has the same central character as $\tilde{\pi}_{\infty} \cdot |\lambda_0|^{\frac{n(n+1)}{4}}$.
\end{proof}

Therefore, by \cite[Theorem A]{kret-shin}, the cuspidal automorphic representation $\tilde\pi$ of $\gsp{2n}{\mb A_{\rats}}$ (for $n \equiv 0,3 \mod 4$) has an associated weakly compatible system of Galois representations
\[ \rho_{\tilde\pi,l}: \galQ \to \gspin{2n+1}{\bQl}
\]
satisfying the conclusions of \textit{loc.cit.} \footnote{Such a compatible system appears in the cohomology of a Shimura variety for a suitable inner form of $\GSp{2n}$ by the construction of \cite{kret-shin}.} In particular, since the central character of $\tilde\pi$ is trivial, \cite[Theorem A,(i)]{kret-shin} implies that $\rho_{\tilde\pi,l}$ is in fact valued in $\spin{2n+1}{\bQl}$. Moreover, for sufficiently general $\{m_i\}_{1 \leq i \leq n}$, $\tilde\pi_{\infty}$ is spin-regular in the sense of \cite{kret-shin}, i.e. the L--parameter of $\tilde\pi_{\infty}$ maps to a regular parameter for $\GL{2^n}$ by the spin representation. In this case, the Zariski-closure of $\rho_{\tilde\pi,l}$ equals $\Spin{2n+1}$, and thus we obtain a $\Spin{2n+1}$--compatible system of motivic Galois representations with full monodromy. 

In contrast, the following lemma shows that for the remaining $n$, we do not expect the Galois representations associated to cuspidal automorphic representations of $\mr{PSp}_{2n}(\mb A_{\rats})$ that are regular algebraic at infinity to land in $\Spin{2n+1}$ by \cite[Conjecture 5.3.4]{buzzard-gee}. 

\begin{lem}\label{C-algebraic}
Suppose $n \equiv 1,2 \mod 4$ (equivalently, $\frac{n(n+1)}{2}$ is odd) and $\pi$ is a cuspidal automorphic representation of $\sp{2n}{\mb A_{\rats}}$ with trivial central character. Denote by $\pi^{\flat}$ the associated representation of $\mr{PSp}_{2n}(\mb A_{\rats})$. If $\pi$ is regular algebraic, \footnote{Since the dual group $\SO{2n+1}$ is adjoint, L--algebraicity is equivalent to C--algebraicity.} then $\pi^{\flat}$ is C--algebraic but not L--algebraic. 
\end{lem}
\begin{proof}
We can write the L--parameter $\phi_{\infty}$ of $\pi_{\infty}$ as 
\[ \phi_{\infty}(z)=diag(z^{m_1}\bar z^{l_1}, ... ,z^{m_n}\bar z^{l_n},1,z^{-m_n}\bar z^{-l_n}, ... ,z^{-m_1}\bar z^{-l_1})
\] with all $m_i, l_i \in \ints$ and $z \in \cmplx^{\times}$.
Arthur’s classification of automorphic representations of classical
groups (\cite{art:end}) implies that $\pi$ has a functorial transfer to a cuspidal automorphic representation of $\gl{2n+1}{\mb A_{\rats}}$, so Clozel's archimedean purity theorem (\cite[Lemme 4.9]{clo:mot}) implies that $m_i+l_i=0$ for all $i$. Since the integers $m_i$ are distinct and nonzero, it is then easy to see that $\phi_{\infty}$ is isomorphic to the L--parameter under the same name defined before Lemma \ref{L-parameters}. Since the L--parameter $\tilde\phi_{\infty}$ of $\pi^{\flat}_{\infty}$ lifts $\phi_{\infty}$, the calculations preceding Lemma \ref{L-parameters} imply that $\sum m_i \equiv \frac{n(n+1)}{2} \equiv 1 \mod 2$. A easy calculation shows that $\rho^{\vee}$ (the half-sum of coroots of $\Spin{2n+1}$) equals $n\lambda_1+(n-1)\lambda_2+\cdots+1\lambda_n$ (which does not lie in $X_{\bullet}(T)$ since $\frac{n(n+1)}{2}$ is odd). It follows that $\sum m_i\lambda_i \in \rho^{\vee}+X_{\bullet}(T)$ and hence $\pi^{\flat}_{\infty}$ is C--algebraic but not L--algebraic ($\tilde\phi_{\infty}(z)=z^{\sum m_i\lambda_i}\bar z^{-\sum m_i\lambda_i}$). 
\end{proof}

However, if $\pi$ is not assumed to be regular at infinity, then $\pi^{\flat}$ can be L--algebraic, in which case the expected Galois representations associated to it lands in $\Spin{2n+1}$ (\cite[Conjecture 3.2.1]{buzzard-gee}).

\end{document}